\RequirePackage{fix-cm}
\documentclass[smallextended,envcountsame]{svjour3} 
\smartqed  
\usepackage{graphicx, amssymb}
\usepackage{times,a4wide,mathrsfs}
\usepackage{amsmath}
\usepackage{amsfonts}

\newcommand{\C}{\mathbb{C}}

\newcommand{\QQ}{\mathbb{Q}}
\newcommand{\NN}{\mathbb{N}}
\newcommand{\PP}{\mathbb{P}}
\newcommand{\LLL}{\mathbb{L}}

\newcommand{\MM}{\mathcal M}

\newcommand{\rom}{\romannumeral}

\makeatletter
\newcommand*{\da@rightarrow}{\mathchar"0\hexnumber@\symAMSa 4B }
\newcommand*{\da@leftarrow}{\mathchar"0\hexnumber@\symAMSa 4C }
\newcommand*{\xdashrightarrow}[2][]{%
  \mathrel{%
    \mathpalette{\da@xarrow{#1}{#2}{}\da@rightarrow{\,}{}}{}%
  }%
}
\newcommand{\xdashleftarrow}[2][]{%
  \mathrel{%
    \mathpalette{\da@xarrow{#1}{#2}\da@leftarrow{}{}{\,}}{}%
  }%
}
\newcommand*{\da@xarrow}[7]{%
  \sbox0{$\ifx#7\scriptstyle\scriptscriptstyle\else\scriptstyle\fi#5#1#6\m@th$}%
  \sbox2{$\ifx#7\scriptstyle\scriptscriptstyle\else\scriptstyle\fi#5#2#6\m@th$}%
  \sbox4{$#7\dabar@\m@th$}%
  \dimen@=\wd0 %
  \ifdim\wd2 >\dimen@
    \dimen@=\wd2 %
  \fi
  \count@=2 %
  \def\da@bars{\dabar@\dabar@}%
  \@whiledim\count@\wd4<\dimen@\do{%
    \advance\count@\@ne
    \expandafter\def\expandafter\da@bars\expandafter{%
      \da@bars
      \dabar@ 
    }%
  }%
  \mathrel{#3}%
  \mathrel{%
    \mathop{\da@bars}\limits
    \ifx\\#1\\%
    \else
      _{\copy0}%
    \fi
    \ifx\\#2\\%
    \else
      ^{\copy2}%
    \fi
  }%
  \mathrel{#4}%
}
\makeatother

\newcommand\undermat[2]{
  \makebox[0pt][l]{$\smash{\underbrace{\phantom{
    \begin{matrix}#2\end{matrix}}}_{\text{$#1$}}}$}#2}

\newtheorem{convention}{Conventions}

\newtheorem{nonumbering}{Theorem}

\newtheorem{nonumberingp}{Proposition}

 \journalname{Annali dell'Universita di Ferrara}

\begin{document}

\title{Some elementary examples of quartics with finite--dimensional motive}

\author{Robert Laterveer}

\institute{CNRS - IRMA, Universit\'e de Strasbourg \at
              7 rue Ren\'e Descartes \\
              67084 Strasbourg cedex\\
              France\\
              \email{laterv@math.unistra.fr}   }

\date{Received: date / Accepted: date}

\maketitle

\begin{abstract}
This small note contains some easy examples of quartic hypersurfaces that have finite--dimensional motive. As an illustration, we verify a conjecture of Voevodsky (concerning smash--equivalence) for some of these special quartics.
 \end{abstract}



\section{Introduction}

The notion of finite--dimensional motive, developed independently by Kimura and O'Sullivan \cite{Kim}, \cite{An}, \cite{MNP}, \cite{J4}, \cite{Iv} has given important new impetus to the study of algebraic cycles. To give but one example: thanks to this notion, we now know the Bloch conjecture is true for surfaces of geometric genus zero that are rationally dominated by a product of curves \cite{Kim}. It thus seems worthwhile to find concrete examples of varieties that have finite--dimensional motive, this being (at present) one of the sole means of arriving at a satisfactory understanding of Chow groups. 

The present note aims to contribute something to the list of examples of varieties with finite--dimensional motive, by considering quartic hypersurfaces.
In any dimension, there is one famous quartic known to have finite--dimensional motive: the Fermat quartic
  \[ (x_0)^4+(x_1)^4+\cdots+(x_{n+1})^4=0\ .\]
  The Fermat quartic has finite--dimensional motive because it is rationally dominated by a product of curves, and the indeterminacy locus is again of Fermat type \cite{Sh}. 
  The main result of this note presents, in any dimension, a family of quartics (containing the Fermat quartic) with finite--dimensional motive:
  
  \begin{nonumbering}[=theorem \ref{main}] The following quartics have finite--dimensional motive:

\noindent
(\rom1) a smooth quartic $X\subset\PP^{3k-1}(\C)$ given by an equation
  \[  f_1(x_0,x_1,x_2)+f_2(x_3,x_4,x_5)+\cdots + f_k(x_{3k-3},x_{3k-2},x_{3k-1})  =0\ ,\]
  where the $f_i$ define smooth quartic curves;
  
  \noindent
  (\rom2) a smooth quartic $X\subset\PP^{3k}(\C)$ given by an equation
   \[ f_1(x_0,x_1,x_2)+f_2(x_3,x_4,x_5)+\cdots + f_k(x_{3k-3},x_{3k-2},x_{3k-1}) +(x_{3k})^4 =0\ ,\]
  where the $f_i$ define smooth quartic curves;
   
  \noindent
  (\rom3) a smooth quartic $X\subset\PP^{3k+1}(\C)$ given by an equation 
   \[ f_1(x_0,x_1,x_2)+ f_2(x_3,x_4,x_5)+\cdots + f_k(x_{3k-3},x_{3k-2},x_{3k-1}) +h(x_{3k},x_{3k+1}) =0\ ,\]
  where the $f_i$ define smooth quartic curves;   
   \end{nonumbering}  

The proof is an elementary argument, based on the fact that the inductive structure exhibited by Shioda \cite{Sh}, \cite{KS} still applies to this kind of hypersurfaces.

To illustrate how nicely finite--dimensionality allows to understand algebraic cycles, we provide an application to a conjecture of Voevodsky's \cite{Voe}:

\begin{nonumberingp}[=proposition \ref{smash}] Let $X$ be a quartic as in theorem \ref{main}, and suppose $X$ has odd dimension. Then numerical equivalence and smash--equivalence coincide for all algebraic cycles on $X$.
\end{nonumberingp}

(For the definition of smash--equivalence, cf. definition \ref{sm}.)




\vskip0.4cm

\begin{convention} All varieties will be projective irreducible varieties over $\C$.

For smooth $X$ of dimension $n$, we will denote by $A^j(X)=A_{n-j}(X)$ the Chow group $CH^j(X)\otimes{\QQ}$ of codimension $j$ cycles under rational equivalence. The notations $A^j_{hom}(X)$ and $A^j_{AJ}(X)$ will denote the subgroup of homologically trivial resp. Abel--Jacobi trivial cycles. The category $\MM_{\rm rat}$ will denote the (contravariant) category of Chow motives with $\QQ$--coefficients over $\C$.
For a smooth projective variety, $h(X)=(X,\Delta_X,0)$ will denote its motive in $\MM_{\rm rat}$. 

The group $H^\ast(X)$ will denote singular cohomology with $\QQ$--coefficients.





\end{convention}

\section{Finite--dimensionality}

We refer to \cite{Kim}, \cite{An}, \cite{MNP}, \cite{Iv}, \cite{J4} for the definition of finite--dimensional motive. 
An essential property of varieties with finite--dimensional motive is embodied by the nilpotence theorem:

\begin{theorem}[Kimura \cite{Kim}]\label{nilp} Let $X$ be a smooth projective variety of dimension $n$ with finite--dimensional motive. Let $\Gamma\in A^n(X\times X)_{}$ be a correspondence which is numerically trivial. Then there is $N\in\NN$ such that
     \[ \Gamma^{\circ N}=0\ \ \ \ \in A^n(X\times X)_{}\ .\]
\end{theorem}

 Actually, the nilpotence property (for all powers of $X$) could serve as an alternative definition of finite--dimensional motive, as shown by a result of Jannsen \cite[Corollary 3.9]{J4}.
 
 \begin{conjecture}[Kimura \cite{Kim}]\label{findim}
  All smooth projective varieties have finite--dimensional motive.
  \end{conjecture}
  
   We are still far from knowing this, but at least there are quite a few non--trivial examples:
 
\begin{remark} 
The following varieties have finite--dimensional motive: abelian varieties, varieties dominated by products of curves \cite{Kim}, $K3$ surfaces with Picard number $19$ or $20$ \cite{P}, surfaces not of general type with $p_g=0$ \cite[Theorem 2.11]{GP}, certain surfaces of general type with $p_g=0$ \cite{GP}, \cite{PW}, \cite{V8}, Hilbert schemes of surfaces known to have finite--dimensional motive \cite{CM}, generalized Kummer varieties \cite[Remark 2.9(\rom2)]{Xu},
 threefolds with nef tangent bundle \cite{Iy} (an alternative proof is given in \cite[Example 3.16]{V3}), fourfolds with nef tangent bundle \cite{Iy2}, log--homogeneous varieties in the sense of \cite{Br} (this follows from \cite[Theorem 4.4]{Iy2}), certain threefolds of general type \cite[Section 8]{V5}, varieties of dimension $\le 3$ rationally dominated by products of curves \cite[Example 3.15]{V3}, varieties $X$ with $A^i_{AJ}(X)_{}=0$ for all $i$ \cite[Theorem 4]{V2}, products of varieties with finite--dimensional motive \cite{Kim}.
\end{remark}

\begin{remark}
It is an embarassing fact that up till now, all examples of finite-dimensional motives happen to lie in the tensor subcategory generated by Chow motives of curves, i.e. they are ``motives of abelian type'' in the sense of \cite{V3}. On the other hand, there exist many motives that lie outside this subcategory, e.g. the motive of a very general quintic hypersurface in $\PP^3$ \cite[7.6]{Del}.
\end{remark}

\section{Main}

\begin{theorem}\label{main} The following quartics have finite--dimensional motive:

\noindent
(\rom1) a smooth quartic $X\subset\PP^{3k-1}(\C)$ given by an equation
  \[  f_1(x_0,x_1,x_2)+f_2(x_3,x_4,x_5)+\cdots + f_k(x_{3k-3},x_{3k-2},x_{3k-1})  =0\ ,\]
  where the $f_i$ define smooth quartic curves;
  
  \noindent
  (\rom2) a smooth quartic $X\subset\PP^{3k}(\C)$ given by an equation
   \[ f_1(x_0,x_1,x_2)+f_2(x_3,x_4,x_5)+\cdots + f_k(x_{3k-3},x_{3k-2},x_{3k-1}) +(x_{3k})^4 =0\ ,\]
  where the $f_i$ define smooth quartic curves;
   
  \noindent
  (\rom3) a smooth quartic $X\subset\PP^{3k+1}(\C)$ given by an equation 
   \[ f_1(x_0,x_1,x_2)+ f_2(x_3,x_4,x_5)+\cdots + f_k(x_{3k-3},x_{3k-2},x_{3k-1}) +h(x_{3k},x_{3k+1}) =0\ ,\]
  where the $f_i$ define smooth quartic curves;   
      \end{theorem}

  \begin{proof} The proof is by induction, using Shioda's trick in the guise of the following proposition (this is \cite[Remark 1.10]{KS}):
  
  \begin{proposition}[Katsura--Shioda \cite{KS}]\label{sh} Let $Z\subset \PP^{m_1+m_2}$ be a smooth hypersurface of degree $d$ defined by an equation
    \[ g_1(x_0,\ldots,x_{m_1})+g_2(x_{m_1+1},\ldots,x_{m_1+m_2})=0\ .\]
   Let $Z_1$ resp. $Z_2$ be the smooth hypersurfaces of dimension $m_1$ resp. $m_2-1$, defined as
     \[ g_1(x_0,\ldots,x_{m_1})+y^d=0\ ,\]
     resp.
     \[ g_2(x_{m_1+1},\ldots,x_{m_1+m_2})+z^d=0\ .\]   
     Then there exists a dominant rational map
     \[ \phi\colon\ \ Z_1\times Z_2\ \dashrightarrow\ Z\ ,\]
     and the indeterminacy of $\phi$ is resolved by blowing up the locus
     \[  \Bigl( Z_1\cap (y=0)\Bigr)\times \Bigl( Z_2\cap (z=0)\Bigr)\ \subset\ Z_1\times Z_2\ .\]
     \end{proposition}
      
     The induction base is $k=1$. In this case, the quartic defined in (\rom1) has finite--dimensional motive because it is a curve; the quartic surface defined in (\rom2) has finite--dimensional motive because it is dominated by a product of curves \cite[Example 11.3]{vG}; case (\rom3) is OK because any quartic threefold $X$ has
     \[ A^\ast_{AJ}(X)=0 \]
     and as such has finite--dimensional motive \cite[Theorem 4]{V2}.
     
     Next, we suppose the theorem is true for $k-1$. 
     
     Let $X\subset\PP^{3k-1}$ be a quartic as in (\rom1). According to proposition \ref{sh}, $X$ is rationally dominated by the product
     $X_1\times S$, where $S$ is a $K3$ surface as in \cite[Example 11.3]{vG}, and $X_1$ is of type
      \[ f_1(x_0,x_1,x_2)+f_2(x_3,x_4,x_5)+\cdots + f_{k-1}(x_{3k-6},x_{3k-5},x_{3k-4}) +y^4 =0\ .\] 
      We have seen that $S$ has finite--dimensional motive, and $X_1$ has finite--dimensional motive by induction. $X$ is dominated by the blow--up of $X_1\times S$ with center
      $Z_1\times C$, where $C\subset S$ is a curve, and $Z_1\subset X_1$ is of type
      \[       f_1(x_0,x_1,x_2)+f_2(x_3,x_4,x_5)+\cdots + f_{k-1}(x_{3k-6},x_{3k-5},x_{3k-4}) =0\ .\]
      The product $Z_1\times C$ has finite--dimensional motive (by induction), hence so has the blow--up and so has $X$.
      
      Let $X\subset\PP^{3k}$ be a quartic as in (\rom2). Now $X$ is rationally dominated by $X_1\times X_2$, where $X_1$ is of type
      \[   f(x_0,x_1,x_2)+f_2(x_3,x_4,x_5)+\cdots + f_{k-1}(x_{3k-6},x_{3k-5},x_{3k-4}) +y^4 =0\ ,\]    
      and $X_2$ is a fourfold of type
      \[   f_k(x_{3k-3},x_{3k-2},x_{3k-1}) +(x_{3k})^4    +z^4=0\ .\]
      Both have finite--dimensional motive by induction. To resolve the indeterminacy, we need to blow up $X_1\times X_2$ in the center of type $Z_1\times Z_2$, where
      $Z_1\subset X_1$ is of type
        \[    f(x_0,x_1,x_2)+f_2(x_3,x_4,x_5)+\cdots + f_{k-1}(x_{3k-6},x_{3k-5},x_{3k-4}) =0\ .\]    
       This $Z_1$ has finite--dimensional motive (by induction), and so does $Z_2$ (any smooth quartic threefold has finite--dimensional motive). It follows that $X$ has finite--dimensional motive.
       
       Lastly, let $X\subset\PP^{3k+1}$ be a quartic as in (\rom3). Then $X$ is rationally dominated by $X_1\times C$, where $X_1$ is of type
       \[  f(x_0,x_1,x_2)+ f_2(x_3,x_4,x_5)+\cdots + f_k(x_{3k-3},x_{3k-2},x_{3k-1}) +y^4 =0\ ,\]   
       and $C$ is a curve. By induction, we may suppose $X_1$ has finite--dimensional motive (because we have already treated case (\rom2)). The indeterminacy locus is of the form $Z_1\times\{\hbox{points}\}$, where $Z_1$ is of type
       \[ f(x_0,x_1,x_2)+ f_2(x_3,x_4,x_5)+\cdots + f_k(x_{3k-3},x_{3k-2},x_{3k-1}) =0\ .\]
       This $Z_1$ also has finite--dimensional motive (because we have already treated case (\rom1)), hence so does the blow--up and so does $X$.
       \end{proof}

\section{Application: Voevodsky's conjecture}

\begin{definition}[Voevodsky \cite{Voe}]\label{sm} Let $X$ be a smooth projective variety. A cycle $a\in A^r(X)$ is called {\em smash--nilpotent\/} 
if there exists $m\in\NN$ such that
  \[ \begin{array}[c]{ccc}  a^m:= &\undermat{(m\hbox{ times})}{a\times\cdots\times a}&=0\ \ \hbox{in}\  A^{mr}(X\times\cdots\times X)_{}\ .
  \end{array}\]
  \vskip0.6cm

Two cycles $a,a^\prime$ are called {\em smash--equivalent\/} if their difference $a-a^\prime$ is smash--nilpotent. We will write $A^r_\otimes(X)\subset A^r(X)$ for the subgroup of smash--nilpotent cycles.
\end{definition}

\begin{conjecture}[Voevodsky \cite{Voe}]\label{voe} Let $X$ be a smooth projective variety. Then
  \[  A^r_{num}(X)\ \subset\ A^r_\otimes(X)\ \ \ \hbox{for\ all\ }r\ .\]
  \end{conjecture}

\begin{remark} It is known \cite[Th\'eor\`eme 3.33]{An} that conjecture \ref{voe} implies (and is strictly stronger than) conjecture \ref{findim}.
\end{remark}

\begin{proposition}\label{smash} Let $X$ be a smooth quartic of the type as in theorem \ref{main}. Suppose the dimension $n$ of $X$ is odd. Then
  \[  A^r_{num}(X)\ \subset\ A^r_\otimes(X)\ \ \ \hbox{for\ all\ }r\ .\]
  \end{proposition}
  
  \begin{proof} As $X$ is a hypersurface, the K\"unneth components are algebraic and the Chow motive of $X$ decomposes
    \[ h(X)=h_n(X)\oplus \bigoplus_j \LLL(n_j)\ \ \ \hbox{in}\ \MM_{\rm rat}\ .\]
    Since $A^r_{num}\bigl(\LLL(n_j)\bigr)=0$, we have
    \[ A^r_{num}(X)= A^r_{num}\bigl(h_n(X)\bigr)\ .\]
    The motive $h_n(X)$ is oddly finite--dimensional. (Indeed, as $n$ is odd we have that the motive $\hbox{Sym}^m h_n(X)\in \MM_{\rm hom}$ is $0$ for some $m>>0$. By finite--dimensionality, the same then holds in $\MM_{\rm rat}$.)
    The proposition now follows from the following result (which is \cite[Proposition 6.1]{Kim}, and which is also applied in \cite{KSeb} where I learned this):
    
    \begin{proposition}[Kimura \cite{Kim}] Let $M\in\MM_{\rm rat}$ be oddly finite--dimensional. Then
       \[ A^r_{}(M)\ \subset\ A^r_\otimes(M)\ \ \ \hbox{for\ all\ }r\ .\]
     \end{proposition}
     
     \end{proof}

\vskip0.6cm

\begin{acknowledgements} This note is a belated echo of the Strasbourg 2014---2015 groupe de travail based on the monograph \cite{Vo}. Thanks to all the participants for the pleasant and stimulating atmosphere. 
Many thanks to Yasuyo, Kai and Len for lots of enjoyable post--work ap\'eritifs. 
\end{acknowledgements}


\end{document}